\documentclass{article}

\usepackage{amsmath,amssymb,amsthm}
\usepackage{mathrsfs}
\usepackage{xcolor}
\usepackage[colorlinks=true]{hyperref}

\newcommand{\R}{\mathbb{R}}
\newcommand{\C}{\mathbb{C}}
\newcommand{\N}{\mathbb{N}}
\newcommand{\abs}[1]{\left\vert #1 \right\vert}
\newcommand{\norm}[1]{\left\Vert #1 \right\Vert}
\newcommand{\energy}{\mathcal{E}}
\newcommand{\obs}{\mathcal{Q}}

\newcommand{\co}{\mathrm{co}}

\renewcommand{\geq}{\geqslant}
\renewcommand{\leq}{\leqslant}

\newtheorem{theorem}{Theorem}[section]
\newtheorem{proposition}[theorem]{Proposition}
\newtheorem{corollary}[theorem]{Corollary}
\newtheorem{lemma}[theorem]{Lemma}

\theoremstyle{definition}
\newtheorem{remark}[theorem]{Remark}

\title{Moving localized observations and Ces\`aro asymptotic observability for conservative PDEs}

\author{Maarten V. de Hoop \footnote{Simons Chair in Computational and Applied Mathematics and Earth Science, TX 77005, USA (\texttt{mvd2@rice.edu})}
\and
Antti Kykk\"anen, \footnote{Computational applied mathematics and operations research, TX 77005, USA (\texttt{ak272@rice.edu})}
\and
Emmanuel Tr\'elat\footnote{Sorbonne Universit\'e, Universit\'e Paris Cit\'e, CNRS, Inria, Laboratoire Jacques-Louis Lions, LJLL, F-75005 Paris, France (\texttt{emmanuel.trelat@sorbonne-universite.fr}).}}

\date{}

\begin{document}

\maketitle

\begin{abstract}
We develop a deterministic large-time mechanism yielding Ces\`aro asymptotic observability inequalities from moving localized observations for conservative evolutions.
On each observation interval, exact convexification on a compact measured homogeneous space replaces full observation on the whole observation manifold by a finite convex combination of translates of one prototype subset.
A switching realization theorem then turns that static design into a genuinely moving observer, while a Hilbertian tail-reduction proposition shows that interval estimates proved only on growing spectral windows still recover the full conserved energy after Ces\`aro averaging.
The resulting design-to-observability chain applies to interior observations for wave, Klein-Gordon, and Schr\"odinger equations on compact measured homogeneous manifolds, to moving boundary caps on the Euclidean ball, and to a singular almost-separated gas-giant boundary model.
The framework is especially relevant when each instantaneous observation set is too small for one to expect a finite-time GCC or time-dependent GCC statement.
\end{abstract}



\section{Introduction}\label{sec_introduction}
Classical observability theory for hyperbolic equations is usually formulated at one prescribed finite time and from one prescribed observation set.
For boundary observations of the wave equation, the Geometric Control Condition (GCC) of \cite{BLR1992} gives a sharp sufficient criterion for such a finite-time inequality: every generalized bicharacteristic must meet the observation region within the observation time.
When the observation set is allowed to move, the relevant microlocal criterion becomes the time-dependent GCC introduced in \cite{LRLTT2017}.
These results are geometric, deterministic, and finite-time in nature.
Their purpose is to recover the full energy from one observation horizon.

The present paper addresses a different problem.
Suppose that the observer is allowed to move and that one is willing to concatenate infinitely many observation time intervals instead of working at one prescribed finite horizon.
Can one then recover the conserved energy from a sequence of localized measurements, even when each instantaneous observation set is too small for one to expect a GCC-based statement?
Our answer is yes for a broad class of conservative models, provided one can design and realize, on each observation interval, a suitable finite-dimensional moving observation protocol.
The resulting inequality is deterministic and large-time, but it is not an exact finite-time observability inequality.
It is a Ces\`aro asymptotic observability inequality obtained from a growing family of moving localized observations.

This point of view is related to, but different from, the spectral observability program initiated in \cite{PTZ2016} and further developed in \cite{HPT2019, HPT2022}.
There the observation domain is fixed and one studies the asymptotic behavior of observation constants as the observation time tends to $+\infty$.
Thus the large-time problem is not merely a repetition of the finite-time GCC question: it can have its own spectral limit and its own optimal design principles.
The present paper follows the same broad philosophy of looking at large times, but the observation domain is now allowed to move and is itself part of the design.

One source of examples comes from recent work on singular acoustic geometries motivated by gas-giant models, see \cite{CdVDdHT, dHIKM2024, dHKT2026}.
Those papers identify, in a boundary setting, a finite-window observation problem on the boundary manifold and derive the calibration estimates needed for large-time observability.
What is developed here is the general moving-observer mechanism itself: exact convexification on the observation manifold, dynamical realization by switching, and Ces\`aro recovery of the full conserved quantity from growing spectral windows.
This mechanism is independent in scope and is not tied to the gas-giant application; it is designed to apply just as well to regular interior and boundary observation problems.

At a conceptual level, the mechanism has three steps.
First, on a finite-dimensional space of spatial profiles, full observation on the whole observation manifold is replaced exactly by a finite convex combination of localized observations supported on translates of one prototype observation subset.
Second, that static convexified design is realized on one observation interval by a genuinely moving observer.
Third, one concatenates larger and larger spectral windows and recovers the full conserved energy by a Ces\`aro argument.
The first and third steps are abstract and are developed here.
The second one is partly abstract, through a switching realization theorem on finite-dimensional trajectory spaces, and partly model-dependent through the calibration estimates that connect the observation to the conserved energy.

Our first abstract result, Proposition \ref{exact-convexification}, is the exact convexification proposition on compact measured homogeneous spaces.
If $\omega \subset M$ is a reference observation subset of relative measure $L$, and if $E$ is any finite-dimensional space of spatial profiles on $M$, then full observation on $M$, restricted to $E$, can be represented as a finite convex combination of observations on translates of $\omega$, with factor $L$.
At this finite-dimensional stage the coefficient is exactly $L$: no loss is introduced by convexification.

The next abstract step is dynamic.
Proposition \ref{switching-realization} turns an exact convexified design into a genuinely moving observer on a finite-dimensional trajectory space by a piecewise constant switching protocol.
This is the rigorous bridge from a static convex combination to time-dependent observation sets.
Under an additional Lie-group kinematic assumption, Proposition \ref{continuous-realization} shows how to optionally impose continuity and a speed bound $\mathsf{V}$ on the moving observer, at the price of replacing $L$ by $L - \varepsilon(\mathsf{V})$, where the loss is quantified by $\mathsf{V}$ and tends to $0$ as $\mathsf{V} \to +\infty$.

The large-time Hilbertian ingredient is Proposition \ref{truncated-tail-proposition}.
It is stated directly in the truncated-window form that is used throughout the PDE applications.
Here $K_m$ denotes the size of the spectral window retained on the $m$-th observation interval: in the abstract Hilbertian setting below, $\mathcal H_{\leq K_m}$ is the span of the first $K_m$ modal blocks.
The condition $K_m \to +\infty$ means that these controlled windows exhaust the energy space along the sequence of intervals.
Positivity of the quadratic observation forms and a uniform upper bound then allow one to remove the tail and recover the exact asymptotic factor $L$ after Ces\`aro averaging.

The main PDE interface is Theorem \ref{main_thm}.
Its statement is organized so as to display the whole design-to-observability chain.
One starts from one prototype observation subset $\omega$ on the observation manifold.
On each time interval $I_m$, one chooses a finite-dimensional profile space $E_m$ for the truncated trajectories, applies Proposition \ref{exact-convexification} to obtain a finite family of translates of $\omega$, and then uses Proposition \ref{switching-realization} to build a piecewise constant moving observer from that design.
The model-dependent input is the calibration estimate on the whole observation manifold.
Once this is available, Proposition \ref{truncated-tail-proposition} upgrades the interval estimates to a Ces\`aro asymptotic observability inequality.

A central novelty of the paper is that this mechanism yields deterministic large-time observability in regimes where one does not expect a finite-time GCC or time-dependent GCC theorem to hold.
The point is not to compete with the microlocal theory of \cite{BLR1992, LRLTT2017}, which is much stronger whenever it applies.
The point is to address a different regime: moving localized observers, often very small, designed interval by interval over long times.
Section \ref{sec_applications} makes this strategy explicit in several concrete PDE settings.
Subsection \ref{sec_interior} treats moving interior observations for wave, Klein-Gordon, and Schr\"odinger equations on compact measured homogeneous manifolds.
Subsection \ref{sec_boundary} then turns to moving boundary observations in regular separated geometries, including the Euclidean ball, and Subsection \ref{sec_gas_giants} treats the singular almost-separated gas-giant boundary model.
In all these examples, a prescribed finite-time observer may miss some rays, whereas a designed large-time observer still recovers the whole conserved quantity.

The conservative character of the dynamics is essential.
The mechanism is tailored to evolutions whose natural energy is preserved, or equivalently propagated by a unitary group.
When one passes from one observation interval to the next, the datum has not lost its high-frequency content through dissipation.
Later measurements may therefore recover modal components that were poorly seen earlier.
This is precisely the effect exploited by the Ces\`aro argument.
By contrast, for parabolic equations such as the heat equation, the semigroup is dissipative and smoothing: high frequencies decay exponentially as time grows, so later measurements do not probe the same spectral content with the same weight.
There is nevertheless a useful structural analogy with the Lebeau-Robbiano strategy: in both settings one first proves estimates on increasing spectral windows and then passes from those windowed estimates to a global statement.
Here, however, the analytical mechanism is completely different.
No dissipation, no smoothing, and no telescoping argument are involved; the passage to the whole energy space uses positivity of the observation forms, a uniform upper bound, and Ces\`aro averaging of the uncontrolled tail.
%For clarity, this analogy will be commented on again in Remark \ref{LR-remark} below.

\medskip

The structure of the paper is as follows.
Section \ref{sec_main_thm} states the main theorem in a form adapted to moving localized observations for conservative evolutions and clarifies the meaning of its hypotheses.
Section \ref{sec_convexification} contains the exact convexification proposition, its matrix reformulation, the switching realization result, and its continuous speed-bounded variant.
Section \ref{sec_cesaro} develops the Ces\`aro tail-reduction proposition for growing spectral windows.
Finally, Section \ref{sec_applications} proves the main theorem and explains how the abstract mechanism applies to interior and boundary observations for several conservative PDEs.

\section{Main result for moving localized observations}\label{sec_main_thm}
The theorem below is stated for conservative evolutions, that is, for unitary dynamics on a Hilbert space, because this is the natural framework for the applications considered later.
It is formulated so as to display the two mechanisms that are developed abstractly in the next sections.
The first mechanism, treated in Section \ref{sec_convexification}, is geometric and finite-dimensional: on each observation interval it constructs a moving design from one prototype subset $\omega$ by exact convexification and switching.
The second mechanism, treated in Section \ref{sec_cesaro}, is Hilbertian and large-time: it turns the resulting estimates on growing spectral windows into a Ces\`aro asymptotic observability inequality.

Let $T_0 > 0$ and define the consecutive observation time intervals
\begin{equation*}%\label{def_I_m}
I_m = ((m - 1) T_0, m T_0) \qquad \forall m \in \N^*.
\end{equation*}
Let $(\mathcal{U}(t))_{t \in \R}$ be a strongly continuous unitary group on a complex Hilbert space $\mathcal{H}$, with skew-adjoint generator $\mathcal{A}$.
For every initial datum $z \in \mathcal{H}$, we denote by
$$
Z(t) = \mathcal{U}(t) z
$$
the corresponding mild solution of
$$
Z'(t) = \mathcal{A} Z(t), \qquad Z(0) = z.
$$
The conserved Hilbert energy is
\begin{equation*}%\label{def_energy}
\energy(z) = \norm{z}_{\mathcal{H}}^2.
\end{equation*}

Assume that $(P_k)_{k \in \N^*}$ is a family of pairwise orthogonal projections on $\mathcal{H}$ such that
$$
z = \sum_{k \in \N^*} P_k z
$$
for every $z \in \mathcal{H}$, with convergence in $\mathcal{H}$.
We think of the subspaces $P_k \mathcal{H}$ as the modal blocks of the system, and we set
$$
\energy_k(z) = \norm{P_k z}_{\mathcal{H}}^2, \qquad \mathcal{H}_{\leq K} = \bigoplus_{k \leq K} P_k \mathcal{H}.
$$

Let $(M,\mu)$ be a compact measured homogeneous space, normalized by $\mu(M) = 1$.
Thus there exists a compact group $G$ acting continuously, transitively, and $\mu$-preservingly on $M$.
In the applications, $M$ is the observation manifold: typically a boundary component, a cross-section, or an interior manifold on which the measurements are performed.
Fix a measurable prototype observation subset $\omega \subset M$ and set
\begin{equation*}%\label{def_L_intro}
L = \mu(\omega) \in (0,1].
\end{equation*}

For every datum $z \in \mathcal{H}$, let
$$
q_z : [0,+\infty) \times M \to \C
$$
be a measurable scalar observation output depending linearly on the trajectory $Z(t)$.
Here $Z$ denotes the abstract state in the Hilbert space $\mathcal H$.
When the evolution comes from a PDE, we denote by $u$ the corresponding physical field and we express the output in terms of $u$.
For instance, for wave or Klein-Gordon equations one has $Z = (u,\partial_t u)$, whereas for Schr\"odinger equations one simply has $Z = u$.
Typical examples are a boundary trace $q_z = \partial_\nu u$, an interior field $q_z = u$, or a time derivative $q_z = \partial_t u$.
A finite family of scalar outputs can be treated as well by summing the corresponding quadratic observation forms; we keep a single scalar output here for notational simplicity.

\begin{theorem}\label{main_thm}
Fix a measurable prototype observation subset $\omega \subset M$, an increasing sequence $(K_m)_{m\in\N^*}$ in $\N$ with $K_m \to +\infty$, and a sequence $(\varepsilon_m)_{m\in\N^*}$ in $(0,L)$ with $\varepsilon_m \to 0$.
Assume that:
\begin{itemize}
\item[$(i)$] There exist constants $c_{T_0}, C_{T_0} > 0$ such that, for every $m\in\N^*$,
\begin{align}
& \int_{I_m} \int_M \abs{q_z(t,y)}^2 d\mu(y)\, dt \leq C_{T_0}\, \energy(z)
\qquad \forall z \in \mathcal{H}, \label{fullobs_upper} \\
c_{T_0} \sum_{k \leq K_m} \energy_k(z) \leq & \int_{I_m} \int_M \abs{q_z(t,y)}^2 d\mu(y)\, dt
\qquad\qquad\qquad\quad \forall z \in \mathcal{H}_{\leq K_m}. \label{fullobs_below}
\end{align}
\item[$(ii)$] For every $m\in\N^*$, there exists a finite-dimensional space $E_m \subset \mathscr{C}^0(M,\C)$ such that
\begin{equation*}
\mathcal{F}_m = \left\{ t \mapsto q_z(t,\cdot)\ \middle\vert\ z \in \mathcal{H}_{\leq K_m} \right\}
\end{equation*}
is a finite-dimensional subspace of $\mathscr{C}^0(I_m;E_m)$.
\end{itemize}
Then there exist piecewise constant measurable maps $t \mapsto \omega_m(t) \subset M$, $t \in I_m$, obtained interval by interval from the prototype subset $\omega$ by exact convexification on $E_m$ and switching realization on $\mathcal F_m$, such that
\begin{equation}\label{main-asymptotic-ineq}
\liminf_{N \to +\infty} \frac{1}{N} \sum_{m=1}^N \int_{I_m} \int_{\omega_m(t)} \abs{q_z(t,y)}^2 d\mu(y)\, dt
\geq
L\, c_{T_0}\, \energy(z)
\qquad \forall z \in \mathcal{H}.
\end{equation}
\end{theorem}

\begin{remark}[Explicit interval-by-interval construction]\label{main-construction-remark}
For each $m\in\N^*$, Proposition \ref{exact-convexification} applied to the pair $(E_m,\omega)$ yields an integer $j_m \leq (\dim E_m)^2 + 1$, elements $g_{m,1},\dots,g_{m,j_m} \in G$, and weights $\theta_{m,1},\dots,\theta_{m,j_m} > 0$ with $\sum_{j=1}^{j_m} \theta_{m,j} = 1$ such that
\begin{equation}\label{main-convexification-equality}
\sum_{j=1}^{j_m} \theta_{m,j} \int_{g_{m,j} \cdot \omega} \abs{f(y)}^2 d\mu(y) = L \int_M \abs{f(y)}^2 d\mu(y)
\qquad \forall f \in E_m.
\end{equation}
Then Proposition \ref{switching-realization}, applied to the finite-dimensional trajectory space $\mathcal F_m$ with tolerance $\varepsilon_m$, yields an integer $R_m\in\N^*$ and a partition into consecutive intervals
\begin{equation}\label{def_partition_main}
I_m = \bigcup_{r=1}^{R_m} I_{m,r}, \qquad I_{m,r} = \bigcup_{j=1}^{j_m} I_{m,r,j}, \qquad \abs{I_{m,r,j}} = \theta_{m,j} \abs{I_{m,r}},
\end{equation}
for which the piecewise constant moving observation subset
\begin{equation}\label{def_omega_m}
\omega_m(t) = g_{m,j} \cdot \omega
\qquad \text{for } t \in I_{m,r,j}
\end{equation}
satisfies
\begin{equation}\label{realization-band-ineq}
\int_{I_m} \int_{\omega_m(t)} \abs{q_z(t,y)}^2 d\mu(y)\, dt
\geq
(L - \varepsilon_m) \int_{I_m} \int_M \abs{q_z(t,y)}^2 d\mu(y)\, dt
\qquad \forall z \in \mathcal{H}_{\leq K_m}.
\end{equation}
In other words, on each observation interval one repeatedly visits the translated subsets $g_{m,j} \cdot \omega$ with the prescribed time fractions $\theta_{m,j}$.
Combined with Remark \ref{matrix-design-remark} further, this gives an explicit interval-by-interval algorithm: compute the finite-dimensional convexified design on $E_m$, then realize its weights dynamically on $I_m$ by switching on a sufficiently fine micro-partition.

Under the additional Lie-group kinematic assumption of Proposition \ref{continuous-realization}, one may even replace this piecewise constant observer by a continuous speed-bounded one, at the price of replacing the factor $L$ by $L - \delta_m$ for a controllable loss $\delta_m \to 0$.
\end{remark}

\begin{remark}\label{remark-calibration}
The estimates \eqref{fullobs_upper}-\eqref{fullobs_below} play different roles.
The lower estimate \eqref{fullobs_below} is a calibration statement on the whole observation manifold $M$: it says that full observation on $M$ controls the truncated energy on each interval of length $T_0$.
The upper estimate \eqref{fullobs_upper} is a boundedness statement.
Together with positivity, it is what allows one to absorb the uncontrolled tail in Section \ref{sec_cesaro}.
For scalar traces or boundary outputs, even these calibration estimates are usually nontrivial PDE statements.
By contrast, for interior observation of the full field on all of $M$, the upper bound is immediate and the lower bound is often tautological, or follows directly from the conserved energy identity.
\end{remark}

\begin{remark}\label{remark-terminology-asymptotic}
We shall usually say that \eqref{main-asymptotic-ineq} is a \emph{Ces\`aro asymptotic observability inequality}.
This terminology emphasizes the specific form of the lower bound: one averages the observation energies collected on the successive intervals $I_m$.
In the literature initiated in \cite{PTZ2016} and continued in \cite{HPT2019, HPT2022}, the wording \emph{time-asymptotic observability} refers instead to a limit of the form
$$
\liminf_{T \to +\infty} \frac{1}{T} \int_0^T \int_{\omega(T,t)} \abs{q_z(t,y)}^2 d\mu(y)\, dt,
$$
for one fixed observation set or one prescribed moving observation set.
Our setting is different: the observation protocol itself is designed interval by interval from exact convexification.
For brevity, we shall occasionally shorten the expression to asymptotic observability when no confusion is possible.
\end{remark}

\begin{remark}\label{remark-vs-gcc}
Theorem \ref{main_thm} is of a different nature from the observability inequalities derived from GCC in \cite{BLR1992} and from the time-dependent GCC in \cite{LRLTT2017}.
It does not provide an exact inequality at one prescribed finite time on the whole energy space.
Instead, it yields a deterministic large-time lower bound obtained by averaging the measurements collected on consecutive observation intervals.
This is particularly relevant when each instantaneous observation subset is too small for one to expect GCC, and when a prescribed moving observer is not designed to satisfy the time-dependent GCC either.
The Euclidean-ball and gas-giant boundary examples discussed later make this contrast completely explicit.
\end{remark}

\begin{remark}\label{remark_choice_Km}
The theorem is formulated for one prescribed increasing sequence $(K_m)$.
In the applications in Section \ref{sec_applications}, the interval construction works for every spectral cutoff $K$, so one may choose $(K_m)$ arbitrarily.
Different choices correspond to different compromises between geometry and large-time convergence: faster growth of $K_m$ reaches high frequencies sooner, while slower growth keeps the finite-dimensional designs smaller and the switching protocol simpler.
\end{remark}

Theorem \ref{main_thm} will be proved in Section \ref{sec_proof_main_thm}.
Its proof combines two abstract ingredients, each of which has an independent interest.
Section \ref{sec_convexification} develops the geometric side of the argument: exact convexification on the observation manifold and its dynamical realization by switching, with an optional continuous speed-bounded refinement.
Section \ref{sec_cesaro} develops the large-time Hilbertian side: the Ces\`aro tail-reduction principle that upgrades interval estimates on growing spectral windows to the final Ces\`aro asymptotic observability inequality.

\section{Exact convexification and dynamical realization}\label{sec_convexification}
We now turn to the first abstract ingredient.
In later applications, $M$ will be the manifold on which the measurements live.
For the moment, however, $(M,\mu)$ is simply a compact measured homogeneous space.

Throughout this section, $(M,\mu)$ is a compact Hausdorff space endowed with a Borel probability measure, $G$ is a compact group, and the action of $G$ on $M$ is continuous, transitive, and $\mu$-preserving, that is,
$$
\mu(g \cdot E) = \mu(E)
\qquad \forall g \in G,
\quad \forall \text{ Borel set } E \subset M.
$$

\begin{remark}\label{homogeneous-space-remark}
This standing hypothesis is satisfied by many natural spaces: spheres $\mathbb{S}^n$, tori $\mathbb T^n$, compact Lie groups endowed with Haar measure, real and complex projective spaces, and, more generally, compact homogeneous spaces $G/H$ endowed with their normalized invariant measure.
It is also stable under measure-preserving conjugation.
More precisely, if $M_0$ is any compact manifold carrying a transitive compact-group action and if $\mu$ is any smooth positive probability density on $M_0$, then Moser's theorem \cite{Moser1965} provides a diffeomorphism transporting the invariant reference density of $M_0$ to $\mu$.
Conjugating the original transitive action by that diffeomorphism then turns $(M_0,\mu)$ into a compact measured homogeneous space.
Thus the abstract framework is not tied to the round sphere; it applies to every smooth probability density on every compact homogeneous manifold, and by transport through a diffeomorphism it also applies to every compact manifold that is diffeomorphic to such a homogeneous model.
\end{remark}

We fix a measurable reference observation subset $\omega \subset M$ and set $L = \mu(\omega) \in (0,1]$.

\subsection{Static exact convexification}

\begin{proposition}\label{exact-convexification}
For every finite-dimensional complex vector space $E \subset \mathscr{C}^0(M,\C)$, there exist an integer $J \leq (\dim E)^2 + 1$, elements $g_1,\dots,g_J \in G$, and weights $\theta_1,\dots,\theta_J > 0$ with $\sum_{j=1}^J \theta_j = 1$ such that
\begin{equation}\label{convexification-equality}
\sum_{j=1}^J \theta_j \int_{g_j \cdot \omega} \abs{f(y)}^2 d\mu(y) = L \int_M \abs{f(y)}^2 d\mu(y)
\qquad \forall f \in E.
\end{equation}
\end{proposition}

\begin{proof}
Define the finite-dimensional real vector space
$$
F = \mathrm{Span}_{\R} \{ \abs{f}^2\ \mid\ f \in E \} \subset \mathscr{C}^0(M,\R).
$$
Then $\dim F \leq (\dim E)^2$.
For each $g \in G$, define the linear functional $\Lambda_g \in F'$ by
$$
\Lambda_g(h) = \int_{g \cdot \omega} h(y)\, d\mu(y)
\qquad \forall h \in F.
$$
The map $g \mapsto \Lambda_g$ is continuous, hence its image $\mathcal{K} = \{ \Lambda_g\ \mid\ g \in G \}$ is compact in the finite-dimensional space $F'$.
Therefore its convex hull $\co(\mathcal{K})$ is compact as well.

Let $dg$ denote the normalized Haar probability measure on $G$ and define the averaged functional
$$
\overline{\Lambda}(h) = \int_G \Lambda_g(h)\, dg
\qquad \forall h \in F.
$$
Since $\overline{\Lambda}$ is the barycenter of a probability measure supported on $\mathcal{K}$, one has $\overline{\Lambda} \in \co(\mathcal{K})$.
By Carath\'eodory's theorem, there exist $J \leq \dim F + 1 \leq (\dim E)^2 + 1$, elements $g_1,\dots,g_J \in G$, and weights $\theta_1,\dots,\theta_J > 0$ with $\sum_{j=1}^J \theta_j = 1$ such that
$$
\overline{\Lambda} = \sum_{j=1}^J \theta_j \Lambda_{g_j}.
$$
It remains to identify $\overline{\Lambda}$.
For $h \in F$, define
$$
\overline{h}(y) = \int_G h(g \cdot y)\, dg.
$$
By left invariance of Haar measure, the function $\overline{h}$ is $G$-invariant, hence constant on $M$ by transitivity.
Since $\mu$ is a $G$-invariant probability measure, this constant is equal to $\int_M h(y)\, d\mu(y)$.
Therefore
$$
\overline{\Lambda}(h) = \int_G \int_{g \cdot \omega} h(y)\, d\mu(y)\, dg = \int_\omega \overline{h}(y)\, d\mu(y) = L \int_M h(y)\, d\mu(y).
$$
Applying this identity to $h = \abs{f}^2$ yields \eqref{convexification-equality}.
\end{proof}

\begin{remark}
Proposition \ref{exact-convexification} is exact in the sense that, at this finite-dimensional design stage, there is no $L - \varepsilon$ loss.
Every later loss of the form $L - \delta$ comes either from the kinematic realization of the moving observer or from the spectral tail that remains uncontrolled after finitely many intervals.
It does not come from convexification itself.
\end{remark}

For the constructive remarks used later in the paper, it is convenient to recast Proposition \ref{exact-convexification} in matrix form.

\begin{corollary}\label{matrix-design}
Under the assumptions of Proposition \ref{exact-convexification}, let $d = \dim E$, let $e_1,\dots,e_d$ be an orthonormal basis of $E$ in the complex Hilbert space $L^2(M,\mu;\C)$, and define for every $g \in G$ the Hermitian matrix $\Gamma(g) \in \C^{d \times d}$ by
$$
\Gamma(g)_{ij} = \int_{g \cdot \omega} e_i(y) \overline{e_j(y)}\, d\mu(y)
\qquad \forall i,j \in \{1,\dots,d\}.
$$
Then there exist $J \leq d^2 + 1$, elements $g_1,\dots,g_J \in G$, and weights $\theta_1,\dots,\theta_J > 0$ with $\sum_{j=1}^J \theta_j = 1$ such that
\begin{equation}\label{matrix-design-equality}
\sum_{j=1}^J \theta_j \Gamma(g_j) = L\,\mathrm{Id}_{\C^d}.
\end{equation}
\end{corollary}

\begin{proof}
For $\xi = (\xi_1,\dots,\xi_d) \in \C^d$, set $f_\xi = \sum_{i=1}^d \xi_i e_i$.
By Proposition \ref{exact-convexification}, there exist $J \leq d^2 + 1$, elements $g_1,\dots,g_J \in G$, and weights $\theta_1,\dots,\theta_J > 0$ with $\sum_{j=1}^J \theta_j = 1$ such that
$$
\sum_{j=1}^J \theta_j \int_{g_j \cdot \omega} \abs{f_\xi(y)}^2 d\mu(y) = L \int_M \abs{f_\xi(y)}^2 d\mu(y) = L \sum_{i=1}^d \abs{\xi_i}^2.
$$
The left-hand side is $\sum_{j=1}^J \theta_j \langle \Gamma(g_j)\xi,\xi\rangle$, so \eqref{matrix-design-equality} follows by polarization.
\end{proof}

\begin{remark}[Finite-dimensional algorithmic interface]\label{matrix-design-remark}
Corollary \ref{matrix-design} is the most convenient form for computations.
Once a basis of $E$ has been fixed, the geometric step becomes a finite-dimensional convex feasibility problem: find a finite convex combination of the positive semidefinite matrices $\Gamma(g)$ that equals $L\,\mathrm{Id}_{\C^d}$.
After discretizing the orbit $g \mapsto \Gamma(g)$, one obtains a finite-dimensional approximation that can be treated numerically by convex optimization or semidefinite programming.
Thus the static part of the construction is already algorithmic:
\begin{enumerate}
\item choose the finite-dimensional profile space $E$ corresponding to the spectral window to be observed;
\item compute, or approximate numerically, the matrices $\Gamma(g)$ associated with translated observation subsets $g \cdot \omega$;
\item solve the convex design problem $\sum_j \theta_j \Gamma(g_j) = L\,\mathrm{Id}_{\C^d}$, or a discretized approximation of it.
\end{enumerate}
The output is a finite list of translated subsets $g_j \cdot \omega$ together with weights $\theta_j$.
Propositions \ref{switching-realization} and \ref{continuous-realization} then provide the dynamic implementation: they convert these static weights into an explicit time-sharing protocol, piecewise constant in the first case and continuous with a controlled speed loss in the second.
\end{remark}

\subsection{Dynamic realization by switching}

Proposition \ref{exact-convexification} is purely static.
We now record the dynamical counterpart that realizes one exact convexified design on one observation interval.
This is the bridge from the spatial design to actual moving observation sets.
It is also the result that produces, in Theorem \ref{main_thm}, the piecewise constant subsets $\omega_m(t)$ and the associated subdivision of each interval $I_m$.

For $\nu \in \N^*$ and $V = (V_1,\dots,V_\nu) \in \C^\nu$, we set $\abs{V}_{\C^\nu}^2 = \sum_{\alpha=1}^\nu \abs{V_\alpha}^2$.

\paragraph{Fast switching realization.}
The output of Proposition \ref{exact-convexification} is a static object: a finite list of translated observation subsets $g_j \cdot \omega$ and weights $\theta_j$.
To use it on a time interval, one must convert these weights into time fractions.
The following proposition does this by fast switching.
The interval is subdivided into many small pieces; on each piece the observer visits the sets $g_j \cdot \omega$ with the prescribed proportions $\theta_j$.
Because the trajectory space is finite-dimensional, the relevant observation densities are uniformly continuous in time, and a sufficiently fine micro-partition reproduces the static convexified average up to an arbitrary loss.

\begin{proposition}\label{switching-realization}
Let $I \subset \R$ be a compact interval, let $E \subset \mathscr{C}^0(M,\C)$ be finite-dimensional, and let $\nu \in \N^*$.
Let $\mathcal{F} \subset \mathscr{C}^0(I;E^\nu)$ be finite-dimensional, where $E^\nu = E \times \cdots \times E$.
By Proposition \ref{exact-convexification}, fix an exact convexified design $(g_j,\theta_j)_{1 \leq j \leq J}$ associated with $E$ and $\omega$, that is,
$$
\sum_{j=1}^J \theta_j \int_{g_j \cdot \omega} \abs{f(y)}^2 d\mu(y) = L \int_M \abs{f(y)}^2 d\mu(y)
\qquad \forall f \in E.
$$
Then, for every $\varepsilon > 0$, there exist an integer $R \geq 1$ and a partition into consecutive subintervals:
$$
I = \bigcup_{r=1}^R I_r, \qquad I_r = \bigcup_{j=1}^J I_{r,j},
\qquad \abs{I_{r,j}} = \theta_j \abs{I_r} ,
%\qquad \forall r \in \{1,\dots,R\}, \quad \forall j \in \{1,\dots,J\},
$$
such that the piecewise constant moving subset defined by $\omega_\varepsilon(t) = g_j \cdot \omega$ for $t \in I_{r,j}$ satisfies
\begin{equation}\label{switching-realization-ineq}
\int_I \int_{\omega_\varepsilon(t)} \abs{V(t,y)}_{\C^\nu}^2 d\mu(y)\, dt
\geq
(L - \varepsilon) \int_I \int_M \abs{V(t,y)}_{\C^\nu}^2 d\mu(y)\, dt
\qquad \forall V \in \mathcal{F}.
\end{equation}
\end{proposition}

\begin{proof}
Let
\begin{equation}\label{defS}
\mathscr{S} = \left\{ V \in \mathcal{F}\ \mid\ \int_I \int_M \abs{V(t,y)}_{\C^\nu}^2\, d\mu(y)\, dt = 1 \right\}.
\end{equation}
Since $\mathcal{F}$ is finite-dimensional, the set $\mathscr{S}$ is compact in $\mathscr{C}^0(I;E^\nu)$.
For every $V \in \mathscr{S}$ and every $j \in \{1,\dots,J\}$, define
\begin{equation}\label{defFVM}
F_{V,j}(t) = \int_{g_j \cdot \omega} \abs{V(t,y)}_{\C^\nu}^2 d\mu(y),
\qquad
F_{V,M}(t) = \int_M \abs{V(t,y)}_{\C^\nu}^2 d\mu(y).
\end{equation}
The families $(F_{V,j})_{V \in \mathscr{S}}$ and $(F_{V,M})_{V \in \mathscr{S}}$ are equicontinuous on $I$.
Choose a partition $I = \bigcup_{r=1}^R I_r$ so fine that, on each $I_r$, the oscillation of every $F_{V,j}$ and every $F_{V,M}$ is at most $\eta$, where $\eta > 0$ will be fixed later.
Write $\tau_r = \abs{I_r}$, choose $t_r \in I_r$, and split each $I_r$ into consecutive subintervals $I_{r,1},\dots,I_{r,J}$ with $\abs{I_{r,j}} = \theta_j \tau_r$.
Set $\omega_\varepsilon(t) = g_j \cdot \omega$ for $t \in I_{r,j}$.

Fix $V \in \mathscr{S}$.
For each $r$,
$$
\sum_{j=1}^J \int_{I_{r,j}} F_{V,j}(t)\, dt \geq \tau_r \sum_{j=1}^J \theta_j F_{V,j}(t_r) - \eta \tau_r.
$$
Because the exact convexification identity holds componentwise for each of the $\nu$ components of $V(t_r,\cdot)$, one has $\sum_{j=1}^J \theta_j F_{V,j}(t_r) = L\, F_{V,M}(t_r)$.
Therefore
$$
\sum_{j=1}^J \int_{I_{r,j}} F_{V,j}(t)\, dt
\geq
L \tau_r F_{V,M}(t_r) - \eta \tau_r
\geq
L \int_{I_r} F_{V,M}(t)\, dt - (L + 1)\eta \tau_r.
$$
Summing over $r$ gives
$$
\int_I \int_{\omega_\varepsilon(t)} \abs{V(t,y)}_{\C^\nu}^2 d\mu(y)\, dt
\geq
L \int_I \int_M \abs{V(t,y)}_{\C^\nu}^2 d\mu(y)\, dt - (L + 1)\eta \abs{I}.
$$
Since $V \in \mathscr{S}$, the right-hand side is $L - (L + 1)\eta \abs{I}$.
Choosing $\eta > 0$ such that $(L + 1)\eta \abs{I} \leq \varepsilon$ proves \eqref{switching-realization-ineq} on $\mathscr{S}$, hence on all of $\mathcal{F}$ by homogeneity.
\end{proof}

\begin{remark}\label{switching-realization-remark}
Proposition \ref{switching-realization} does not construct the exact convexified design itself.
That design is the static object produced by Proposition \ref{exact-convexification}, namely the finite family of translates $(g_j \cdot \omega)$ together with the weights $(\theta_j)$.
What Proposition \ref{switching-realization} adds is the time variable: it subdivides the observation interval into many smaller pieces and visits those translates with the prescribed time fractions.
This is why the resulting observer is piecewise constant and may switch many times on the interval.
The proof is constructive once the design has been computed.
Indeed, after fixing a basis of the finite-dimensional trajectory space $\mathcal F$, one can estimate a common modulus of continuity for the associated quadratic forms $F_{V,j}$ and $F_{V,M}$, choose a mesh size accordingly, and then build the switching partition explicitly.
\end{remark}

\paragraph{Continuous realization with bounded speed.}
The switching protocol of Proposition \ref{switching-realization} allows instantaneous jumps between the translated subsets $g_j \cdot \omega$.
This is the right abstract model for the interval-by-interval construction, but in some applications one may want to impose kinematic constraints on the observer.
For instance, a physical sensor may have to move continuously and with bounded speed along the observation manifold.
The next result shows that such constraints can indeed be incorporated, at the price of a controlled loss from $L$ to $L - \varepsilon$.

\begin{proposition}\label{continuous-realization}
Assume in addition that $G$ is a compact connected Lie group endowed with a right-invariant Riemannian metric, and let $d_G$ denote the corresponding distance.
Under the hypotheses of Proposition \ref{switching-realization}, fix a piecewise smooth cycle visiting the design points $g_1,\dots,g_J$ in that order, and let
$D = \sum_{j=1}^{J-1} d_G(g_j,g_{j+1}) + d_G(g_J,g_1)$
be its total length.
For the compact set $\mathscr{S}$ defined by \eqref{defS} and the functions $F_{V,M}$ and $F_{V,j}$ defined by \eqref{defFVM}, define the common modulus of continuity
$$
\varpi_{\mathcal F}(\tau) = \sup_{V \in \mathscr S} \sup_{\substack{t,s \in I \\ \abs{t-s} \leq \tau}} \max \left(
\abs{F_{V,M}(t) - F_{V,M}(s)}, \max_{1 \leq j \leq J} \abs{F_{V,j}(t) - F_{V,j}(s)} \right).
$$
Then, for every speed bound $\mathsf{V} > D / \abs{I}$ and every $R\in\N^*$ satisfying $R < \frac{\mathsf{V} \abs{I}}{D}$, there exists a continuous piecewise smooth path $g_{\mathsf{V},R} : I \to G$ such that%
\footnote{Fix a bi-invariant Riemannian metric on $G$, and denote by $\vert \cdot \vert_G$ the induced norm on tangent vectors.
Equivalently, one may fix an $\mathrm{Ad}$-invariant Euclidean norm on the Lie algebra $\mathfrak g$ and identify each tangent space $T_g G$ with $\mathfrak g$ by right translations.
Since $G$ is compact, any two such choices are uniformly equivalent, so the bounded-speed condition \eqref{bounded_speed} is intrinsic up to renormalizing the constant $\mathsf{V}$.}
\begin{equation}\label{bounded_speed}
\vert\dot g_{\mathsf{V},R}(t)\vert_G \leq \mathsf{V} \qquad \text{for a.e. } t \in I,
\end{equation}
and, setting $\omega_{\mathsf{V},R}(t) = g_{\mathsf{V},R}(t) \cdot \omega$, one has
\begin{equation}\label{continuous-realization-ineq}
\int_I \int_{\omega_{\mathsf{V},R}(t)} \abs{V(t,y)}_{\C^\nu}^2 d\mu(y)\, dt \geq \big( L - \varepsilon_{\mathcal F}(\mathsf{V},R) \big) \int_I \int_M \abs{V(t,y)}_{\C^\nu}^2 d\mu(y)\, dt
\qquad \forall V \in \mathcal F,
\end{equation}
where
\begin{equation}\label{continuous-realization-error}
\varepsilon_{\mathcal F}(\mathsf{V},R) = L \frac{D R}{\mathsf{V} \abs{I}} + (L + 1) \abs{I}\, \varpi_{\mathcal F}\left( \frac{\abs{I}}{R} \right).
\end{equation}
In particular, there exists a function $\mathsf{V} \mapsto \varepsilon_{\mathcal F}(\mathsf{V})$ with $\varepsilon_{\mathcal F}(\mathsf{V}) \to 0$ as $\mathsf{V} \to +\infty$ such that \eqref{continuous-realization-ineq} holds with $\varepsilon_{\mathcal F}(\mathsf{V})$ in place of $\varepsilon_{\mathcal F}(\mathsf{V},R)$.
\end{proposition}

\begin{proof}
Fix $\mathsf{V} > D/\abs{I}$ and $R < \mathsf{V} \abs{I}/D$.
Partition $I$ into $R$ consecutive intervals $I_r$ of equal length $\tau = \abs{I}/R$.
Within each $I_r$, reserve a total transition time $D/\mathsf{V}$ and a total dwelling time $\tau - D/\mathsf{V}$.
During the dwelling time, keep the observer successively at the fixed positions $g_1 \cdot \omega,\dots,g_J \cdot \omega$ for time fractions $\theta_1,\dots,\theta_J$ of $\tau - D/\mathsf{V}$.
Between two successive dwelling phases, move continuously along the chosen cycle at speed exactly $\mathsf{V}$.
This constructs a continuous piecewise smooth path $g_{\mathsf{V},R}$ with the announced speed bound.

Fix $V \in \mathscr S$.
Arguing as in the proof of Proposition \ref{switching-realization}, but ignoring the transition pieces, one gets on each $I_r$
$$
\int_{I_r} \int_{\omega_{\mathsf{V},R}(t)} \abs{V(t,y)}_{\C^\nu}^2 d\mu(y)\, dt \geq L \left( 1 - \frac{D}{\mathsf{V} \tau} \right) \int_{I_r} F_{V,M}(t)\, dt - (L + 1) \tau\, \varpi_{\mathcal F}(\tau).
$$
Summing over $r$ yields
$$
\int_I \int_{\omega_{\mathsf{V},R}(t)} \abs{V(t,y)}_{\C^\nu}^2 d\mu(y)\, dt \geq L \left( 1 - \frac{D R}{\mathsf{V} \abs{I}} \right) \int_I \int_M \abs{V(t,y)}_{\C^\nu}^2 d\mu(y)\, dt - (L + 1)\abs{I}\, \varpi_{\mathcal F}(\tau).
$$
Since $V \in \mathscr S$, this is exactly \eqref{continuous-realization-ineq}-\eqref{continuous-realization-error}.
Finally, since $\varpi_{\mathcal F}(\tau) \to 0$ as $\tau \to 0$, one can choose an integer $R(\mathsf{V})$ tending to $+\infty$ and satisfying $R(\mathsf{V}) = \mathrm{o}(\mathsf{V})$.
Then both terms in \eqref{continuous-realization-error} tend to $0$, which proves the last statement.
\end{proof}

\begin{remark}
Proposition \ref{continuous-realization} makes explicit the trade-off between kinematics and observability.
The first term in \eqref{continuous-realization-error} is the time fraction lost to the transitions and is directly proportional to the inverse speed bound $\mathsf{V}^{-1}$.
The second term measures the time oscillation of the finite-dimensional trajectory family and is controlled by refining the micro-partition.
This is the precise sense in which imposing continuity or a speed constraint naturally replaces the exact factor $L$ by $L - \varepsilon$.
Here again the procedure is constructive: once one has chosen a cycle through the design points and a speed bound $\mathsf{V}$, the error estimate \eqref{continuous-realization-error} gives an explicit admissible refinement parameter $R$.
\end{remark}

\begin{remark}[Interpretation as a sensor flight path]
In geometric sensing applications one may interpret Proposition \ref{continuous-realization} as a bounded-speed flight-path condition.
Assume, for instance, that $M$ is a Riemannian observation surface, that the action of $G$ on $M$ is smooth, and that the prototype set $\omega$ is a sensor footprint centered at some point $y_0 \in M$.
Then the path $\gamma_{\mathsf V,R}(t) = g_{\mathsf V,R}(t) \cdot y_0$ is the ground track of the sensor center, and the moving observation set is the corresponding transported footprint $g_{\mathsf V,R}(t) \cdot \omega$.
Since $G$ and $M$ are compact in the applications considered here, the differential of the action is uniformly bounded.
Thus the group-speed constraint $\vert \dot g_{\mathsf V,R}(t) \vert_G \leq \mathsf V$ implies a physical speed bound of the form $\vert \dot \gamma_{\mathsf V,R}(t) \vert_h \leq C_{\rm act}\, \mathsf V$ for a.e. $t$, where $C_{\rm act}$ depends only on the action and on the chosen metrics.
In this sense, a prescribed maximal flight speed can be incorporated by choosing $\mathsf V$ accordingly.
Additional constraints, such as a preferred direction of travel or a monotone longitudinal drift, would amount to restricting the admissible cycles through the design points; they are not part of Proposition \ref{continuous-realization}, but the same estimate can be applied to any admissible cycle whose total length $D$ is known.
\end{remark}

\section{Ces\`aro asymptotic observability from growing spectral windows}\label{sec_cesaro}
We now turn to the second abstract ingredient.
This section is purely Hilbertian.
The index $m$ labels a sequence of nonnegative quadratic forms; in PDE applications it will later label successive observation intervals, but no time variable is needed here.
The guiding principle is that if the $m$-th quadratic form controls a larger and larger spectral window, then Ces\`aro averaging still recovers the full energy.

Let $\mathcal{H}$ be a complex Hilbert space endowed with pairwise orthogonal projections $(P_k)_{k \in \N^*}$ such that
$$
z = \sum_{k \in \N^*} P_k z
\qquad \forall z \in \mathcal{H},
$$
with convergence in $\mathcal{H}$.
We set
$$
\energy_k(z) = \norm{P_k z}_{\mathcal{H}}^2, \qquad \energy(z) = \sum_{k \in \N^*} \energy_k(z) = \norm{z}_{\mathcal{H}}^2
\qquad \forall z \in \mathcal{H}.
$$
Let $(\obs_m)_{m \in \N^*}$ be a sequence of nonnegative continuous quadratic forms on $\mathcal{H}$, and define their Ces\`aro averages by
$$
\mathcal{A}_N(z) = \frac{1}{N} \sum_{m=1}^N \obs_m(z)
\qquad \forall z \in \mathcal{H}, \quad \forall N \in \N^*.
$$

\begin{proposition}\label{truncated-tail-proposition}
Let $L \in (0,1]$.
Assume that there exist constants $c, C > 0$, an increasing sequence $(K_m)_{m\in\N^*}$ in $\N$ with $K_m \to +\infty$, and a sequence $(\varepsilon_m)_{m\in\N^*}$ in $[0,L)$ with $\varepsilon_m \to 0$ such that
\begin{align}
& \obs_m(z) \leq C\, \energy(z)
\qquad\qquad\quad \forall m\in\N^*, \quad \forall z \in \mathcal{H}, \label{truncated-tail-upper} \\
& \obs_m(z) \geq c (L - \varepsilon_m)\, \energy(z)
\qquad \forall m\in\N^*, \quad \forall z \in \mathcal{H}_{\leq K_m}. \label{truncated-tail-lower}
\end{align}
Then
\begin{equation}\label{truncated-tail-asymptotic}
\liminf_{N \to +\infty} \mathcal{A}_N(z) \geq L\, c\, \energy(z)
\qquad \forall z \in \mathcal{H}.
\end{equation}
\end{proposition}

\begin{proof}
Fix $z \in \mathcal{H}$ and $\eta \in (0,1)$.
For every $m\in\N^*$, write
$$
z = z_m^{\leq K_m} + z_m^{>K_m},
\qquad
z_m^{\leq K_m} = \sum_{k \leq K_m} P_k z,
\qquad
z_m^{>K_m} = \sum_{k > K_m} P_k z.
$$
Since $\obs_m$ is a nonnegative continuous quadratic form on $\mathcal{H}$, there exists a bounded nonnegative selfadjoint operator $A_m$ on $\mathcal{H}$ such that $\obs_m(w) = \langle A_m w, w \rangle_{\mathcal{H}}$ for every $w \in \mathcal{H}$. Hence
$$
\obs_m(z) = \obs_m(z_m^{\leq K_m} + z_m^{>K_m}) \geq \big( \obs_m(z_m^{\leq K_m})^{1/2} - \obs_m(z_m^{>K_m})^{1/2} \big)^2.
$$
Using $2ab \leq \eta a^2 + \eta^{-1} b^2$, we infer that
$$
\obs_m(z) \geq (1 - \eta) \obs_m(z_m^{\leq K_m}) - (\eta^{-1} - 1) \obs_m(z_m^{>K_m}).
$$
By \eqref{truncated-tail-upper} and \eqref{truncated-tail-lower},
$$
\obs_m(z)
\geq
(1 - \eta) c (L - \varepsilon_m) \sum_{k \leq K_m} \energy_k(z)
-
(\eta^{-1} - 1) C\, r_m(z),
\qquad
r_m(z) = \sum_{k > K_m} \energy_k(z).
$$
Since $K_m \to +\infty$, one has $r_m(z) \to 0$.
Therefore the Ces\`aro averages of $(r_m(z))$ converge to $0$, and the same holds for $(\varepsilon_m)$.
Averaging the previous inequality from $m = 1$ to $N$ and letting $N \to +\infty$ yields
$$
\liminf_{N \to +\infty} \mathcal{A}_N(z)
\geq
(1 - \eta) c L\, \energy(z).
$$
Letting $\eta \to 0$ gives \eqref{truncated-tail-asymptotic}.
\end{proof}

\begin{remark}\label{LR-remark}
Proposition \ref{truncated-tail-proposition} has a formal resemblance with the Lebeau-Robbiano strategy for parabolic control (see \cite{LR1995}).
In both settings one first proves estimates on spectral windows that grow with the iteration index and then passes from those windowed estimates to a global statement.
The similarity is only organizational, however.
The parabolic Lebeau-Robbiano method relies on dissipation, smoothing, and a telescoping argument across frequency scales.
Here the evolution is conservative, no smoothing occurs, and the passage to the full energy space uses only positivity of the observation forms, the uniform upper bound \eqref{truncated-tail-upper}, and Ces\`aro averaging of the uncontrolled tail.
\end{remark}

\section{Applications to conservative PDEs}\label{sec_applications}
We now return to PDE models.
The common pattern is always the same:
\begin{itemize}
\item[$(a)$] identify, on each spectral window, a finite-dimensional trajectory space of observation profiles;
\item[$(b)$] apply exact convexification on the corresponding profile space;
\item[$(c)$] realize that design by switching on the observation interval;
\item[$(d)$] combine the resulting interval estimates by Proposition \ref{truncated-tail-proposition}.
\end{itemize}
We begin by proving Theorem \ref{main_thm} itself.

\subsection{Proof of Theorem \ref{main_thm}}\label{sec_proof_main_thm}

\begin{proof}[Proof of Theorem \ref{main_thm}]
Fix $m\in\N^*$.
By assumption $(ii)$, the observation profiles generated on $I_m$ by data in $\mathcal H_{\leq K_m}$ form a finite-dimensional subspace $\mathcal F_m = \left\{ t \mapsto q_z(t,\cdot)\ \middle\vert\ z \in \mathcal H_{\leq K_m} \right\} \subset \mathscr C^0(I_m;E_m)$.

We now perform explicitly the two geometric steps.
First, apply Proposition \ref{exact-convexification} to the pair $(E_m,\omega)$.
This yields an integer $j_m \leq (\dim E_m)^2 + 1$, elements $g_{m,1},\dots,g_{m,j_m}\in G$, and weights $\theta_{m,1},\dots,\theta_{m,j_m} > 0$ with $\sum_{j=1}^{j_m} \theta_{m,j} = 1$ such that \eqref{main-convexification-equality} holds.
Second, apply Proposition \ref{switching-realization} to the interval $I_m$, to the finite-dimensional trajectory space $\mathcal F_m$, and to the above exact convexified design, with tolerance $\varepsilon_m$.
This produces an integer $R_m \in \N^*$, a subdivision \eqref{def_partition_main} and a piecewise constant moving subset \eqref{def_omega_m} such that \eqref{realization-band-ineq} holds.

Define
$$
\obs_m(z) = \int_{I_m} \int_{\omega_m(t)} \abs{q_z(t,y)}^2 d\mu(y)\, dt
\qquad \forall z \in \mathcal H.
$$
Since $\omega_m(t) \subset M$, assumption \eqref{fullobs_upper} gives
$$
\obs_m(z) \leq \int_{I_m} \int_M \abs{q_z(t,y)}^2 d\mu(y)\, dt \leq C_{T_0}\, \energy(z)
\qquad \forall z \in \mathcal H.
$$
If $z \in \mathcal H_{\leq K_m}$, then the switching estimate together with \eqref{fullobs_below} yields
$$
\obs_m(z)
\geq
(L - \varepsilon_m) \int_{I_m} \int_M \abs{q_z(t,y)}^2 d\mu(y)\, dt
\geq
c_{T_0}(L - \varepsilon_m)\, \energy(z).
$$
Therefore the sequence $(\obs_m)_{m\in\N^*}$ satisfies the assumptions of Proposition \ref{truncated-tail-proposition} with $c = c_{T_0}$ and $C = C_{T_0}$.
The conclusion \eqref{main-asymptotic-ineq} follows.
\end{proof}

The proof shows exactly where the two abstract ingredients enter.
Section \ref{sec_convexification} builds the moving observer on each interval from one prototype subset $\omega$, while Section \ref{sec_cesaro} turns the resulting growing-window estimates into the final Ces\`aro asymptotic observability inequality.

\subsection{Interior moving observations on compact measured homogeneous manifolds}\label{sec_interior}
We first treat interior observations, where the observation manifold is the ambient manifold itself.
Throughout this subsection, $(M,h)$ is a compact smooth Riemannian manifold without boundary and $n = \dim M$.
We assume that the normalized Riemannian volume $\mu = \frac{dV_h}{\int_M dV_h}$ turns $(M,\mu)$ into a compact measured homogeneous space in the sense of Section \ref{sec_convexification}.
This includes every compact homogeneous manifold endowed with any smooth positive probability density.
More generally, if $M$ is diffeomorphic to a compact homogeneous manifold $M_0$, then transporting $\mu$ to $M_0$, applying Moser's theorem there, and conjugating back the transitive action shows that $(M,\mu)$ again fits the same framework.
Thus the measured-homogeneous hypothesis is much more flexible than the round-sphere model and applies well beyond spheres.
We denote by $\Delta_h$ the Laplace-Beltrami operator associated with $h$, with the sign convention that $-\Delta_h$ is nonnegative.

\subsubsection{Wave and Klein-Gordon equations}

Consider the wave/Klein-Gordon equation
\begin{equation}\label{wave-klein-gordon-equation}
\partial_t^2 u - \Delta_h u + \mathfrak{m}^2 u = 0,
\qquad
u\vert_{t=0} = u_0,
\qquad
\partial_t u\vert_{t=0} = u_1,
\end{equation}
where $\mathfrak{m} \in \R$ and where $(u_0,u_1)$ belongs to the natural finite-energy space
$$
\mathcal H_{\mathrm{wave}} =
\begin{cases}
H^1(M;\C) \times L^2(M;\C) & \text{if } \mathfrak m \neq 0,\\
\big( H^1(M;\C) / \C \big) \times L^2(M;\C) & \text{if } \mathfrak m = 0.
\end{cases}
$$
The conserved energy is
$$
\energy(u_0,u_1)
=
\int_M \big( \abs{u_1(y)}^2 + \abs{\nabla_h u_0(y)}_h^2 + \mathfrak{m}^2 \abs{u_0(y)}^2 \big)\, dV_h(y).
$$

Let $(\phi_k)_{k\in\N^*}$ be an orthonormal basis of $L^2(M;\C)$ made of Laplace eigenfunctions, $-\Delta_h \phi_k = \lambda_k \phi_k$, and for $K\in\N^*$ set $E_K = \mathrm{Span}\{ \phi_1,\dots,\phi_K \}$.
Let $\mathcal{H}_{\leq K}$ denote the corresponding truncated finite-energy space, and let
$$
\rho_k = \sqrt{\lambda_k + \mathfrak{m}^2}
\qquad \forall k\in\N^*.
$$

To apply the abstract moving-observer mechanism with the kinetic observation $\abs{\partial_t u}^2$, one needs a calibration estimate on the whole manifold $M$.
The next lemma provides this ingredient.

\begin{lemma}%[Full-manifold kinetic calibration]
\label{kinetic-calibration}
For every $T_0 > 0$, there exist constants $0 < c_{T_0}^{\mathrm{kin}} \leq C_{T_0}^{\mathrm{kin}}$ such that, for every interval $I \subset \R$ of length $T_0$ and every finite-energy solution of \eqref{wave-klein-gordon-equation},
\begin{equation*}%\label{wave-kinetic-calibration}
c_{T_0}^{\mathrm{kin}}\, \energy(u_0,u_1)
\leq
\int_I \int_M \abs{\partial_t u(t,y)}^2 dV_h(y)\, dt
\leq
C_{T_0}^{\mathrm{kin}}\, \energy(u_0,u_1).
\end{equation*}
\end{lemma}

\begin{proof}
Any finite-energy solution $u$ of \eqref{wave-klein-gordon-equation} can be expanded as
$$
u(t,y) = \sum_{k\in\N^*} \left( a_k \cos(\rho_k t) + b_k \frac{\sin(\rho_k t)}{\rho_k} \right) \phi_k(y),
$$
with
$$
\energy(u_0,u_1) = \sum_{k\in\N^*} \big( \rho_k^2 \abs{a_k}^2 + \abs{b_k}^2 \big).
$$
Using orthogonality in space,
$$
\int_I \int_M \abs{\partial_t u(t,y)}^2 dV_h(y)\, dt
=
\sum_{k\in\N^*} \int_I \abs{ - \rho_k a_k \sin(\rho_k t) + b_k \cos(\rho_k t)}^2 dt.
$$
For every fixed $\rho \geq 0$, the right-hand side is the quadratic form associated with the Gram matrix of $\sin(\rho t)$ and $\cos(\rho t)$ on an interval of length $T_0$, applied to the vector $(\rho a,b)$.
If $\rho > 0$, the smallest eigenvalue of that Gram matrix is $\frac{T_0}{2} - \frac{\abs{\sin(\rho T_0)}}{2\rho} > 0$ and it tends to $T_0/2$ as $\rho \to +\infty$.
Therefore its infimum on the discrete set $\{ \rho_k > 0 \}$ is positive.
If $\rho_k = 0$, which may occur only for the wave equation on the constant mode, then the corresponding contribution is simply $T_0 \abs{b_k}^2$, while the displacement coefficient $a_k$ carries no energy.
This proves the lower bound.
The upper bound is immediate from
$\int_I \abs{ - \rho_k a_k \sin(\rho_k t) + b_k \cos(\rho_k t)}^2 dt \leq T_0 \big( \rho_k^2 \abs{a_k}^2 + \abs{b_k}^2 \big)$.
\end{proof}

With Lemma \ref{kinetic-calibration} in hand, one may now run the convexification-plus-switching-plus-tail-reduction scheme for the kinetic observation.

\begin{theorem}\label{wave-kinetic-theorem}
Fix $T_0 > 0$, let $\omega \subset M$ be measurable with $\mu(\omega) = L$, let $(K_m)_{m\in\N^*}$ be any increasing sequence in $\N$ such that $K_m \to +\infty$, and let $(\varepsilon_m)_{m\in\N^*}$ be any sequence in $(0,L)$ with $\varepsilon_m \to 0$.
Then, for each $m\in\N^*$, there exists a finite exact convexified design for $(E_{K_m},\omega)$ given by Proposition \ref{exact-convexification}, and Proposition \ref{switching-realization} applied to the corresponding finite-dimensional trajectory space on $I_m$ yields a piecewise constant measurable map $t \mapsto \omega_m(t) \subset M$, constant on a partition of the form \eqref{def_partition_main} and taking only values among the translates selected by that design, such that the corresponding finite-energy solution of \eqref{wave-klein-gordon-equation} with initial datum $(u_0,u_1) \in \mathcal H_{\mathrm{wave}}$ satisfies
\begin{equation}\label{wave-kinetic-asymptotic}
\liminf_{N \to +\infty}
\frac{1}{N}
\sum_{m=1}^N
\int_{I_m} \int_{\omega_m(t)} \abs{\partial_t u(t,y)}^2 dV_h(y)\, dt
\geq
L\, c_{T_0}^{\mathrm{kin}}\, \energy(u_0,u_1).
\end{equation}
\end{theorem}

\begin{proof}
For each $m\in\N^*$, let $\mathcal{F}_m$ be the finite-dimensional trajectory space on $I_m$ formed by the scalar outputs $t \mapsto \partial_t u(t,\cdot)$ when $(u_0,u_1)$ ranges over $\mathcal{H}_{\leq K_m}$.
Then $\mathcal{F}_m \subset \mathscr{C}^0(I_m;E_{K_m})$ is finite-dimensional.
Apply Proposition \ref{exact-convexification} to $E_{K_m}$ and then Proposition \ref{switching-realization} to $\mathcal{F}_m$ with the chosen tolerance $\varepsilon_m$.
This yields piecewise constant moving subsets $t \mapsto \omega_m(t)$ on $I_m$ such that
$$
\int_{I_m} \int_{\omega_m(t)} \abs{\partial_t u(t,y)}^2 dV_h(y)\, dt \geq (L - \varepsilon_m) \int_{I_m} \int_M \abs{\partial_t u(t,y)}^2 dV_h(y)\, dt
$$
for every solution with initial data in $\mathcal H_{\leq K_m}$.
Using Lemma \ref{kinetic-calibration}, one gets
$$
\int_{I_m} \int_{\omega_m(t)} \abs{\partial_t u(t,y)}^2 dV_h(y)\, dt \geq (L - \varepsilon_m) c_{T_0}^{\mathrm{kin}}\, \energy(u_0,u_1) 
\qquad \forall (u_0,u_1) \in \mathcal H_{\leq K_m}.
$$
The upper bound
$$
\int_{I_m} \int_{\omega_m(t)} \abs{\partial_t u(t,y)}^2 dV_h(y)\, dt \leq C_{T_0}^{\mathrm{kin}}\, \energy(u_0,u_1)
$$
holds for all data because $\omega_m(t) \subset M$ and Lemma \ref{kinetic-calibration} holds on the whole manifold.
Proposition \ref{truncated-tail-proposition} then yields \eqref{wave-kinetic-asymptotic}.
\end{proof}

\begin{remark}\label{wave-Km-choice-remark}
In Theorem \ref{wave-kinetic-theorem}, the choice of the sequence $(K_m)$ is completely free.
The same observation will apply to the other examples below.
A faster growth of $K_m$ makes the high-frequency tail disappear sooner in the Ces\`aro argument, but it also increases the dimension of the convexified design and therefore the switching complexity on each interval.
More precisely, Carath\'eodory's theorem gives at most $(\dim E_{K_m})^2 + 1$ distinct observation positions on the interval $I_m$, while Proposition \ref{switching-realization} introduces an additional micro-partition whose size depends on a modulus of continuity of the finite-dimensional trajectory space.
Thus the geometric complexity grows at least quadratically with $\dim E_{K_m}$ and also deteriorates when one asks for smaller tolerances $\varepsilon_m$.
\end{remark}

\begin{remark}[Algorithmic construction of the sets $\omega_m(t)$]\label{wave-algorithmic-remark}
The construction in Theorem \ref{wave-kinetic-theorem} is completely explicit once the eigenfunctions of $-\Delta_h$ are available.
For each $m$:
\begin{enumerate}
\item choose the spectral size $K_m$ and form the finite-dimensional profile space $E_{K_m}$;
\item compute the matrices $\Gamma(g)$ of Corollary \ref{matrix-design} for a basis of that space and solve the finite-dimensional convex design problem
$$
\sum_j \theta_{m,j} \Gamma(g_{m,j}) = L\,\mathrm{Id}_{\C^{\dim E_{K_m}}};
$$
\item apply Proposition \ref{switching-realization} to the resulting static design and to the finite-dimensional trajectory space on $I_m$, in order to build the micro-partition and the piecewise constant path $t \mapsto \omega_m(t)$.
\end{enumerate}
This produces a piecewise constant, typically highly oscillatory, moving observer.
If one prefers a continuous path or a speed-bounded sensor, Proposition \ref{continuous-realization} gives a quantitative replacement of the factor $L$ by $L - \varepsilon$ in terms of the admissible speed.
\end{remark}

\begin{remark}[Novelty of the spherical and toric examples]\label{wave-sphere-novelty-remark}
Theorem \ref{wave-kinetic-theorem} gives a new Ces\`aro asymptotic observability inequality in several regimes where classical finite-time geometry does not provide a usual observability result.

First, take $M = \mathbb{S}^n$ with the round metric and let $\omega$ be a small geodesic cap.
For fixed-domain interior observation of the wave equation, such a cap lies outside the classical finite-time GCC regime as soon as it is too small, because there exist geodesics that never meet it.
Likewise, a prescribed moving cap need not satisfy the time-dependent GCC of \cite{LRLTT2017}.
Theorem \ref{wave-kinetic-theorem} nevertheless yields deterministic Ces\`aro asymptotic observability from a sequence of moving caps designed interval by interval from $\omega$.

Second, the same conclusion holds on the flat torus $\mathbb T^n$ for any small translated observation subset $\omega$.
Again, fixed-domain GCC fails for sufficiently small subsets because straight geodesics can avoid them, while the present construction gives a deterministic large-time moving-observation inequality.

More generally, the same conclusion holds on every compact manifold that can be identified, through a diffeomorphism, with a compact homogeneous model carrying a smooth positive probability density.
After transport to that homogeneous model and Moser conjugation (see \cite{Moser1965}), one obtains the same measured-homogeneous framework, and the moving observation sets are simply the pullbacks of the prototype translates.
On a nonhomogeneous representative they need not be geodesic caps or literal translations for the metric $h$.

This large-time result is of a different nature from the finite-time moving-ball observability statements derived in \cite{Letrouit2021} under geodesic recurrence assumptions.
Here the observer is not prescribed in advance to satisfy a microlocal geometric criterion.
It is designed interval by interval from exact convexification and switching.
\end{remark}

\subsubsection{Schr\"odinger equations}

Consider the free Schr\"odinger equation
\begin{equation*}%\label{schrodinger-equation}
i \partial_t u - \Delta_h u = 0, \qquad u\vert_{t=0} = u_0 \in L^2(M;\C).
\end{equation*}
For fixed observation domains, internal observability and control for the Schr\"odinger equation on compact manifolds are classical, going back to \cite{Lebeau1992} (see also Laurent \cite{Laurent2014} for a concise survey). On flat tori, observability from any nonempty open subset is known by several methods, see in particular \cite{AnantharamanMacia2014, BurqZworski2012}.

Here the conserved quantity is simply
$$
\energy(u_0) = \norm{u_0}_{L^2(M;\C)}^2,
$$
and the full-manifold calibration is the exact identity
\begin{equation}\label{schrodinger-full-manifold}
\int_I \int_M \abs{u(t,y)}^2 dV_h(y)\, dt = T_0 \norm{u_0}_{L^2(M;\C)}^2
\end{equation}
for every interval $I$ of length $T_0$.
Thus the Schr\"odinger proof is formally the same as for Theorem \ref{wave-kinetic-theorem}, but simpler: the observation has only one scalar component and the full-manifold calibration is the obvious exact identity \eqref{schrodinger-full-manifold}, whereas the wave/Klein-Gordon case required the nontrivial kinetic calibration of Lemma \ref{kinetic-calibration}.

\begin{theorem}\label{schrodinger-local-mass-theorem}
Fix $T_0 > 0$, let $\omega \subset M$ be measurable with $\mu(\omega) = L$, let $(K_m)_{m\in\N^*}$ be any increasing sequence in $\N$ such that $K_m \to +\infty$, and let $(\varepsilon_m)_{m\in\N^*}$ be any sequence in $(0,L)$ with $\varepsilon_m \to 0$.
Then, for each $m\in\N^*$, there exists a finite exact convexified design for $(E_{K_m},\omega)$ given by Proposition \ref{exact-convexification}, and Proposition \ref{switching-realization} applied to the associated Schr\"odinger trajectory space on $I_m$ yields a piecewise constant measurable map $t \mapsto \omega_m(t) \subset M$, taking only values among the translates selected by that design, such that every $u_0 \in L^2(M;\C)$ satisfies
\begin{equation}\label{schrodinger-asymptotic-moving}
\liminf_{N \to +\infty}
\frac{1}{N}
\sum_{m=1}^N
\int_{I_m} \int_{\omega_m(t)} \abs{u(t,y)}^2 dV_h(y)\, dt
\geq
L\, T_0 \norm{u_0}_{L^2(M;\C)}^2.
\end{equation}
\end{theorem}

\begin{proof}
Fix $m\in\N^*$.
Let $\mathcal{F}_m$ be the finite-dimensional trajectory space on $I_m$ formed by the solutions $u(t,\cdot)$ when $u_0$ ranges over $E_{K_m}$.
Then $\mathcal{F}_m \subset \mathscr{C}^0(I_m;E_{K_m})$ is finite-dimensional.
Apply Proposition \ref{switching-realization} to $\mathcal{F}_m$ and to the exact convexified design associated with $E_{K_m}$.
This yields a piecewise constant moving subset $t \mapsto \omega_m(t)$ on $I_m$ such that
$$
\int_{I_m} \int_{\omega_m(t)} \abs{u(t,y)}^2 dV_h(y)\, dt
\geq
(L - \varepsilon_m) \int_{I_m} \int_M \abs{u(t,y)}^2 dV_h(y)\, dt
$$
for every solution with $u_0 \in E_{K_m}$.
Using \eqref{schrodinger-full-manifold} gives
$$
\int_{I_m} \int_{\omega_m(t)} \abs{u(t,y)}^2 dV_h(y)\, dt
\geq
(L - \varepsilon_m) T_0 \norm{u_0}_{L^2(M;\C)}^2
\qquad \forall u_0 \in E_{K_m}.
$$
The upper bound
$$
\int_{I_m} \int_{\omega_m(t)} \abs{u(t,y)}^2 dV_h(y)\, dt \leq T_0 \norm{u_0}_{L^2(M;\C)}^2
$$
holds for all data because $\omega_m(t) \subset M$.
Proposition \ref{truncated-tail-proposition} then yields \eqref{schrodinger-asymptotic-moving}.
\end{proof}

\begin{remark}
Theorem \ref{schrodinger-local-mass-theorem} is obtained by exactly the same mechanism as the wave/Klein-Gordon results above: exact convexification on one finite spectral window, switching realization on the associated finite-dimensional trajectory space, and tail reduction.
The only simplification is that the full-manifold calibration is the trivial identity \eqref{schrodinger-full-manifold}.
\end{remark}

\begin{remark}\label{sphere-schrodinger-novelty}
As an example, take $M = \mathbb{S}^2$ with the round metric and let $\omega$ be a sufficiently small polar cap.
For fixed observation domains, finite-time Schr\"odinger observability on $\mathbb{S}^2$ may fail for such caps because highest-weight spherical harmonics concentrate near the equator, see \cite{MaciaRiviere2019}.
Theorem \ref{schrodinger-local-mass-theorem} shows that, by moving the cap according to the convexified interval-by-interval design, one nevertheless recovers the full $L^2$ mass in the deterministic Ces\`aro sense \eqref{schrodinger-asymptotic-moving}.
This provides a concrete large-time observability statement on a positively curved compact manifold beyond the available fixed-domain theory.

On flat tori, fixed-domain Schr\"odinger observability from any nonempty open subset is known and is much stronger than the present asymptotic statement, see \cite{AnantharamanMacia2014, BurqZworski2012, Jaffard1990}; on compact hyperbolic surfaces, see \cite{Jin2018}.
For these geometries, the point is therefore not to improve the known fixed-domain observability results.
Rather, Theorem \ref{schrodinger-local-mass-theorem} shows that the same convexification-switching-tail-reduction mechanism also produces an explicit moving-observer design in settings where stronger fixed-domain theorems are already available.
On the sphere, by contrast, the moving-cap conclusion is new.
\end{remark}

\subsection{Boundary moving observations}\label{sec_boundary}
We next turn to boundary observations.
For readability we formulate the results for one distinguished boundary component, but the same argument applies to one fixed transverse section as well.
Thus one may keep in mind either a boundary manifold $M \simeq \partial X$ or, more generally, an interior section $\Sigma_a \simeq M$ at some fixed normal coordinate $x = a$.
The geometric contrast with the interior subsection above remains the same: now the observation manifold is not the whole ambient manifold but one codimension-one copy of $M$.

\subsubsection{Separated boundary theorem}
For $\ell\in(0,+\infty]$, let
$$
X = [0,\ell]_x \times M_y,
$$
where $(M,\mu)$ is a compact measured homogeneous space.
Let $(\mathcal U(t))_{t \in \R}$ be the strongly continuous unitary group on a Hilbert energy space $\mathcal H$ generated by the conservative PDE under consideration on $X$.
For each datum $z \in \mathcal H$, let $q_z(t,y)$ be the chosen scalar boundary output measured on one distinguished copy of $M$, for instance on $\{0\} \times M$, on $\{\ell\} \times M$, or on one fixed transverse section $\{a\} \times M$ for some $a>0$.

Assume that the evolution admits a separated modal decomposition
$$
\mathcal H = \bigoplus_{k\in\N^*} \mathcal H_k,
$$
associated with an orthonormal basis $(\phi_k)_{k\in\N^*}$ of eigenfunctions of $-\Delta_h$ on $M$.
For $K \in \N^*$, set
$$
E_K = \mathrm{Span}\{ \phi_1,\dots,\phi_K \},
\qquad
\mathcal H_{\leq K} = \bigoplus_{k \leq K} \mathcal H_k.
$$
Assume that, for every $K$, the traces $q_z(t,\cdot)$ associated with data $z \in \mathcal H_{\leq K}$ belong to $E_K$ for every $t$.
We also assume that there exist $T_0 > 0$ and constants $0 < c_{T_0} \leq C_{T_0}$ such that, for every interval $I \subset \R$ of length $T_0$,
\begin{equation}\label{boundary-full-calibration}
c_{T_0}\, \energy(z)
\leq
\int_I \int_M \abs{q_z(t,y)}^2 d\mu(y)\, dt
\leq
C_{T_0}\, \energy(z)
\qquad \forall z \in \mathcal H.
\end{equation}
In separated models, \eqref{boundary-full-calibration} is a one-dimensional normal observability estimate, typically proved after tangential separation by a nonharmonic Fourier argument of Ingham type (see \cite{Ingham1936, KL2005} for the general harmonic method and, e.g., \cite{Mehrenberger2009} for an explicit boundary observability implementation on product domains).

\begin{theorem}\label{boundary-separated-theorem}
Fix a measurable prototype observation subset $\omega \subset M$ with $\mu(\omega) = L$, let $(K_m)_{m\in\N^*}$ be any increasing sequence in $\N$ such that $K_m \to +\infty$, and let $(\varepsilon_m)_{m\in\N^*}$ be any sequence in $(0,L)$ such that $\varepsilon_m \to 0$.
Then, for each $m\in\N^*$, there exists a finite exact convexified design for $(E_{K_m},\omega)$ given by Proposition \ref{exact-convexification}, and Proposition \ref{switching-realization} applied to the corresponding boundary trajectory space on $I_m$ yields a piecewise constant measurable map $t \mapsto \omega_m(t) \subset M$, taking only values among the translates selected by that design, such that
\begin{equation}\label{boundary-separated-asymptotic}
\liminf_{N \to +\infty}
\frac{1}{N}
\sum_{m=1}^N
\int_{I_m} \int_{\omega_m(t)} \abs{q_z(t,y)}^2 d\mu(y)\, dt
\geq
L\, c_{T_0}\, \energy(z)
\qquad \forall z \in \mathcal H.
\end{equation}
\end{theorem}

\begin{proof}
Fix $m\in\N^*$.
Let $\mathcal F_m$ be the finite-dimensional trajectory space on $I_m$ formed by the scalar outputs $q_z(t,\cdot)$ when $z$ ranges over $\mathcal H_{\leq K_m}$.
Then $\mathcal F_m \subset \mathscr{C}^0(I_m;E_{K_m})$ is finite-dimensional.
Apply Proposition \ref{switching-realization} to $\mathcal F_m$ and to the exact convexified design associated with $E_{K_m}$.
This yields a piecewise constant moving subset $t \mapsto \omega_m(t)$ on $I_m$ such that
$$
\int_{I_m} \int_{\omega_m(t)} \abs{q_z(t,y)}^2 d\mu(y)\, dt
\geq
(L - \varepsilon_m)
\int_{I_m} \int_M \abs{q_z(t,y)}^2 d\mu(y)\, dt
$$
for every $z \in \mathcal H_{\leq K_m}$.
Using \eqref{boundary-full-calibration}, one gets
$$
\int_{I_m} \int_{\omega_m(t)} \abs{q_z(t,y)}^2 d\mu(y)\, dt
\geq
(L - \varepsilon_m) c_{T_0}\, \energy(z)
\qquad \forall z \in \mathcal H_{\leq K_m}.
$$
The upper bound
$$
\int_{I_m} \int_{\omega_m(t)} \abs{q_z(t,y)}^2 d\mu(y)\, dt
\leq
C_{T_0}\, \energy(z)
$$
holds for all data because $\omega_m(t) \subset M$.
Proposition \ref{truncated-tail-proposition} then yields \eqref{boundary-separated-asymptotic}.
\end{proof}

\begin{remark}\label{boundary-almost-separated-remark}
Exact separation is used above only to identify, on each spectral window, a finite-dimensional boundary trace space living on the observation manifold $M$.
The same proof would apply without essential change to regular \emph{almost-separated} geometries as soon as one can establish, by perturbation or quasi-normal-form arguments, two analogous ingredients: a full-boundary calibration estimate of the form \eqref{boundary-full-calibration}, and a finite-dimensional trace-space inclusion for the truncated trajectories.
Natural candidates are smooth quasi-product or shell-like perturbations of separated geometries.
Because this robustness analysis is strongly model-dependent, we do not formulate a general theorem here.
The gas-giant example treated in Section \ref{sec_gas_giants} should be viewed as a singular almost-separated instance of this broader philosophy.
\end{remark}

\subsubsection{Wave and Klein-Gordon equations on the Euclidean ball}

\begin{corollary}\label{euclidean-ball-corollary}
Let $X = \overline B(0,1) \subset \R^{n+1}$ be the closed Euclidean unit ball, and consider the Dirichlet wave or Klein-Gordon equation
$$
\partial_t^2 u - \Delta u + \mathfrak{m}^2 u = 0,
\qquad
u\vert_{\partial X} = 0.
$$
Let $q(t,\theta) = \partial_r u(t,1,\theta)$ be the boundary normal derivative on $M = \partial X = \mathbb{S}^n$ endowed with its normalized surface measure $\mu$.
Fix $T_0 > 2$, let $\omega \subset \mathbb S^n$ be a measurable boundary cap with $\mu(\omega) = L$, let $(K_m)_{m\in\N^*}$ be any increasing sequence such that $K_m \to +\infty$, and let $(\varepsilon_m)_{m\in\N^*}$ be any sequence in $(0,L)$ with $\varepsilon_m \to 0$.
Then, for each $m\in\N^*$, one applies Proposition \ref{exact-convexification} to the spherical-harmonic space $E_{K_m}$ on $\mathbb S^n$ and then Proposition \ref{switching-realization} to the corresponding boundary trajectory space on $I_m$.
This yields piecewise constant moving boundary subsets $t \mapsto \omega_m(t) \subset \mathbb S^n$, each taking only rotational copies of $\omega$, such that
\begin{equation*}%\label{euclidean-ball-asymptotic}
\liminf_{N \to +\infty}
\frac{1}{N} \sum_{m=1}^N \int_{I_m} \int_{\omega_m(t)} \abs{\partial_r u(t,1,\theta)}^2 d\mu(\theta)\, dt \geq L\, c_{T_0}\, \energy(u_0,u_1)
\end{equation*}
for every finite-energy initial datum.
\end{corollary}

\begin{proof}
In polar coordinates $(r,\theta)$, the Dirichlet eigenfunctions are $r^{-\frac{n-1}{2}} J_{\nu_\ell}(j_{\nu_\ell,p} r)\, Y_{\ell,m}(\theta)$, where $Y_{\ell,m}$ are spherical harmonics on $\mathbb S^n$, $\nu_\ell = \ell + \frac{n-1}{2}$, and $j_{\nu_\ell,p}$ denotes the $p$-th positive zero of the Bessel function $J_{\nu_\ell}$.
Hence the boundary output $\partial_r u(t,1,\cdot)$ belongs, on each tangential spectral window, to the corresponding finite-dimensional spherical-harmonic space on $\mathbb S^n$.
The full-boundary calibration estimate \eqref{boundary-full-calibration} for every $T_0 > 2$ follows from the classical boundary observability theorem of \cite{BLR1992}, since every generalized bicharacteristic of the Euclidean unit ball meets the full boundary within time $2$.
Therefore Theorem \ref{boundary-separated-theorem} applies.
\end{proof}

\begin{remark}\label{euclidean-ball-novelty-remark}
A fixed small cap on $\partial B(0,1)$ does not satisfy the classical boundary GCC, and a prescribed moving cap need not satisfy the time-dependent GCC either.
Corollary \ref{euclidean-ball-corollary} therefore yields a completely new deterministic Ces\`aro asymptotic boundary observability inequality.
\end{remark}

\subsubsection{Gas-giant boundary observations as a singular almost-separated extension}\label{sec_gas_giants}
We now turn to a singular almost-separated boundary model arising in acoustic geometry near gas-giant surfaces.
The Euclidean-ball corollary above is the regular separated prototype, while the gas-giant geometry developed in \cite{CdVDdHT, dHIKM2024, dHKT2026} is its singular and almost-separated analogue.
The observation manifold is the outer boundary $M = \partial X$, endowed with the normalized surface measure $\mu = \frac{dV_{\partial}}{\int_M dV_{\partial}}$, where $dV_{\partial}$ is the Riemannian surface measure associated with the limiting boundary metric.
Because $M$ is assumed to be diffeomorphic to $\mathbb{S}^n$, Remark \ref{homogeneous-space-remark} provides the measured homogeneous-space structure required for exact convexification.

The gas-giant model is not exactly separated.
The Laplacian is shown in \cite{CdVDdHT} to admit an almost-separated quasi-isometric normal form, and \cite{dHKT2026} derives from that structure the full-boundary calibration estimate and the finite-dimensional tangential trace spaces on growing spectral windows.
What the present paper adds is the missing geometric-dynamical step: once those finite-dimensional trace spaces are known, exact convexification and switching realization produce the moving caps themselves.

\begin{corollary}\label{gas-giant-corollary}
In the gas-giant setting of \cite{dHKT2026}, let $\omega \subset M$ be a prototype boundary cap with relative measure $L = \mu(\omega)$, let $(K_m)_{m\in\N^*}$ be any increasing sequence such that $K_m \to +\infty$, and let $(\varepsilon_m)_{m\in\N^*}$ be any sequence in $(0,L)$ with $\varepsilon_m \to 0$.
Then, for each $m\in\N^*$, Proposition \ref{exact-convexification} applied to the boundary trace space $E_m$ and the prototype cap $\omega$, followed by Proposition \ref{switching-realization} on the associated finite-dimensional trajectory space on $I_m$, yields a piecewise constant moving boundary subset $t \mapsto \omega_m(t) \subset M$ such that
\begin{equation}\label{gas-giant-obs-ineq}
\liminf_{N \to +\infty} \frac{1}{N} \sum_{m=1}^N \int_{I_m} \int_{\omega_m(t)} \abs{q_z(t,y)}^2 d\mu(y)\, dt \geq L\, c_{T_0}\, \energy(z)
\qquad \forall z \in \mathcal H.
\end{equation}
\end{corollary}

\begin{proof}
For each $m\in\N^*$, the analysis of \cite{dHKT2026} provides:
\begin{itemize}
\item[$(a)$] a full-boundary calibration estimate of the form
$$
c_{T_0}\, \energy(z) \leq \int_{I_m} \int_M \abs{q_z(t,y)}^2 d\mu(y)\, dt \leq C_{T_0}\, \energy(z)
$$
for all $z \in \mathcal H_{\leq K_m}$ and the corresponding upper bound for all $z \in \mathcal H$;
\item[$(b)$] a finite-dimensional boundary trace space $E_m \subset \mathscr{C}^0(M,\C)$ such that
$$
\mathcal F_m = \left\{ t \mapsto q_z(t,\cdot)\ \mid\ z \in \mathcal H_{\leq K_m} \right\} \subset \mathscr{C}^0(I_m;E_m)
$$
is finite-dimensional.
\end{itemize}
Apply Proposition \ref{exact-convexification} to $(E_m,\omega)$ and then Proposition \ref{switching-realization} to $\mathcal F_m$ with the tolerance $\varepsilon_m$.
This yields a piecewise constant moving cap $t \mapsto \omega_m(t)$ on $I_m$ such that
$$
\int_{I_m} \int_{\omega_m(t)} \abs{q_z(t,y)}^2 d\mu(y)\, dt
\geq
(L - \varepsilon_m)
\int_{I_m} \int_M \abs{q_z(t,y)}^2 d\mu(y)\, dt
$$
for every $z \in \mathcal H_{\leq K_m}$.
Using the lower full-boundary calibration on $\mathcal H_{\leq K_m}$ gives
$$
\int_{I_m} \int_{\omega_m(t)} \abs{q_z(t,y)}^2 d\mu(y)\, dt
\geq
(L - \varepsilon_m) c_{T_0}\, \energy(z)
\qquad \forall z \in \mathcal H_{\leq K_m},
$$
while the upper bound on the whole boundary gives
$$
\int_{I_m} \int_{\omega_m(t)} \abs{q_z(t,y)}^2 d\mu(y)\, dt
\leq
C_{T_0}\, \energy(z)
\qquad \forall z \in \mathcal H.
$$
Proposition \ref{truncated-tail-proposition} therefore yields \eqref{gas-giant-obs-ineq}.
\end{proof}

\begin{remark}\label{gas-giant-novelty-remark}
Corollary \ref{gas-giant-corollary} is the singular almost-separated member of the family of examples developed in this paper.
It shows that the same finite-window plus switching plus tail-reduction mechanism produces concrete deterministic Ces\`aro asymptotic observability inequalities both in regular separated geometries and in singular almost-separated ones.
The singular metric is not the source of the moving-design mechanism; rather, the novelty of the gas-giant model is that the required calibration and trace-space reduction survive the quasi-isometric almost-separated normal form of \cite{CdVDdHT, dHKT2026}.
In particular, the corollary establishes precisely the moving-cap boundary observability statement announced at the end of \cite{dHKT2026}.
\end{remark}

\subsection{Other conservative systems}\label{sec_other-systems}
The abstract mechanism is not tied to scalar wave equations.
What it really requires is:
\begin{itemize}
\item[$(i)$] a conservative Hilbert dynamics carrying a modal decomposition;
\item[$(ii)$] a compact measured homogeneous observation manifold and a prototype localized observation subset;
\item[$(iii)$] finite-dimensional profile spaces on the growing spectral windows;
\item[$(iv)$] calibration estimates on the whole observation manifold or on the full boundary or section.
\end{itemize}
Whenever these ingredients are available, the conclusion has the same form as above: a deterministic Ces\`aro asymptotic observability inequality with exact factor $L$.

A first family of further examples is provided by beam, plate, and elastic systems on product, warped-product, or almost-separated geometries.
Typical model equations are the Euler-Bernoulli or Timoshenko beam equations, fourth-order plate equations such as $\partial_t^2 u + \Delta^2 u = 0$ on shells of the form $[0,\ell] \times M$, and linear elasticity systems $\partial_t^2 U + \mathcal L U = 0$ on domains with separated or almost-separated tangential variables.
The relevant observation operator is then a boundary trace, a bending moment, a traction or stress component, or another scalar quantity whose low spectral windows give rise to finite-dimensional profile spaces on the observation manifold.
Once the calibration estimates and the interval realization estimate are established, the same exact convexification plus Ces\`aro mechanism yields asymptotic observability from moving localized observations.

\begin{remark}
The comparison with GCC is model-dependent in these further examples.
For beams, plates, shells, or elastic systems, the currently available finite-time geometric theories are much less complete than for the scalar wave equation.
This is another reason why the deterministic large-time mechanism isolated here may be useful: it produces a robust asymptotic conclusion once one can design and realize the interval-by-interval observation protocol.
For neighbouring results with time-dependent or moving observation sets in rather different geometries, see \cite{JamingKomornik2020} for moving or oblique observations of the one-dimensional Schr\"odinger equation and of beam/plate equations, and \cite{WangWang2024} for observability of dispersive equations from line segments on the torus.
For classical fixed-domain control of a higher-order model, one may also compare with the exact internal control result of \cite{Jaffard1990} for rectangular plate vibrations.
\end{remark}

\paragraph{Acknowledgment.}

MVdH was supported by the National Science Foundation under grant DMS-2407456, the Simons Foundation under the MATH + X program and the corporate members of the Geo-Mathematical Imaging Group at Rice University.
ET acknowledges the support of the grant ANR-23-CE40-0010-02 (Einstein-PPF).

\end{document}